\documentclass{amsart}
\usepackage{amsfonts}
\usepackage{mathtools}

\usepackage[bookmarksnumbered,colorlinks, linkcolor=blue, citecolor=red, pagebackref, bookmarks, breaklinks]{hyperref}

\newcommand{\R}{\mathbb{R}}
\newcommand{\N}{\mathbb{N}}

\newcommand{\cd}{\rightharpoonup}
\newcommand{\tog}{\stackrel{\gamma}{\to}}
\newcommand{\toh}{\stackrel{H^c}{\to}}
\def\R{{\mathbb {R}}}
\def\N{{\mathbb {N}}}
\def\C{\mathbf {C}}

\newcommand{\ve}{\varepsilon}
\DeclareMathOperator\supp{supp}
\DeclareMathOperator\diver{div}
\DeclareMathOperator\cp{cap_\Phi}
\def\pint{\operatorname {--\!\!\!\!\!\int\!\!\!\!\!--}}
\newtheorem{teo}{Theorem}[section]
\newtheorem{lema}[teo]{Lemma}
\newtheorem{prop}[teo]{Proposition}

\theoremstyle{remark}
\newtheorem{remark}[teo]{Remark}
\newtheorem{exam}[teo]{Example}
\theoremstyle{definition}
\newtheorem{defi}[teo]{Definition}

\numberwithin{equation}{section}

\title[Continuity of solutions for the $\Delta_\phi$-Laplacian operator]
{Continuity of solutions for the $\Delta_\phi$-Laplacian operator}

\author[N.A. Cantizano, A.M. Salort and J.F Spedaletti]{Natal\'i A. Cantizano, Ariel M. Salort and Juan F. Spedaletti}

\address[N. Cantizano]{
\hfill\break\indent Instituto de Matem\'atica Aplicada San Luis,
IMASL.
 \hfill\break\indent Universidad Nacional de San Luis and CONICET.
 \hfill\break\indent
Ejercito de los Andes 950. San Luis, Argentina.}
 \email{{\tt ncantizano@unsl.edu.ar}
 }

\address[A.M. Salort]{
\hfill\break\indent Departamento de Matem\'atica FCEyN
\hfill\break\indent Universidad de Buenos Aires and IMAS - CONICET.
\hfill\break\indent  Ciudad Universitaria, Pabell\'on I. 
 Buenos Aires, Argentina.}

\email{asalort@dm.uba.ar} \urladdr{http://mate.dm.uba.ar/~asalort/}

\address[J.F. Spedaletti]{
\hfill\break\indent Instituto de Matem\'atica Aplicada San Luis,
IMASL.
 \hfill\break\indent Universidad Nacional de San Luis and CONICET.
 \hfill\break\indent
Ejercito de los Andes 950. San Luis, Argentina.}
\email{jfspedaletti@unsl.edu.ar}
\urladdr{https://www.researchgate.net/profile/Juan{\_}Spedaletti2}


\begin{document}

\begin{abstract}
In this paper we give sufficient conditions to obtain continuity results of solutions for the so called {\em $\phi-$Laplacian} $\Delta_\phi$ with respect to domain perturbations. We  point out that this kind of results can be extended to a more general class of operators including, for instance, nonlocal nonstandard growth type operators.
%These conditions are given in terms of the capacity in the Orlicz-Sobolev context of the approximating domains and in terms of the function $\phi$. 
\end{abstract}

\maketitle
\section{Introduction}

The problem of determining whether solutions of differential equations are stable with respect to perturbations of the domain has become a fundamental task due to its connection with numerical optimal shape design models. The bibliography on this subject is huge. We refer for instance the books   \cite{Allaire, Bucur-Butazzo, HenrotPierre, Pironneau, Sokol} and reference therein.

Consider a sequence of open subsets $\{\Omega_k\}_{k\in\N}$ contained in a fixed and bounded design box $D\subset \R^n$ converging, in some sense, to an open subset $\Omega\subset D$. To determine under which conditions can be guaranteed the convergence of a solution $u_{\Omega_k}$ of some  elliptic problem defined in $\Omega_k$ to a solution $u_\Omega$ of the same equation in $\Omega$, has been a challenging and interesting issue dealt in the last decades.

For most of the usual topologies on the family of subsets, the convergence is not implied. To be more precise, let us focus in the simplest model given by the Dirichlet Laplacian problem. In the seminal paper \cite{Cio-Mu}, Cioranescu and Murat (see also  \cite{Tartar}) showed that when taking $D=[0,1]^2\subset \R^2$ and $\Omega_k=D\setminus \cup_{i,j=1}^{n-1} B_{n^{-2}} (\frac{i}{n},\frac{j}{n})$, then $\Omega_k$ converges to the empty set in the Hausdorff complementary topology (see definition \ref{Hausdorff} for details), however, if $u_{\Omega_k}\in H^1_0(\Omega)$ is solution of
$$
-\Delta u_{\Omega_k} = f \text { in } \Omega, \quad u_{\Omega_k}=0 \text { on } \partial \Omega,
$$
then  $u_{\Omega_k} \to v$ weakly in $H^1_0(D)$ to the solution of the following \emph{homogenized equation} involving a  so-called \emph{strange term}
$$
-\Delta v + \frac{2}{\pi}v= f \text { in } \Omega, \quad v=0 \text { on } \partial \Omega.
$$

In some particular cases the continuity of solutions can be ensured. For instance, when $\{\Omega_k\}_{k\in\N}$ is an increasing sequence of convex polygons such that $\Omega=\cup_k \Omega_k$. This example can be generalized in terms of the capacity of the symmetric difference between $\Omega_k$ and $\Omega$. See \cite{Bucur-Butazzo, Bucur-Trebeschi, HenrotPierre, Sverak} for example. Moreover, the capacitary condition can be get rid of and changed by a simpler one when the number of connected components of the sequence remains uniformly  bounded. This is the well-known \emph{\v{S}ver\'{a}k Theorem} proved in \cite{Sverak}.

This kind of stability results were extended to a more general frameworks involving nonlinear (and/or non-local) operators with $p-$Laplacian type structure. See \cite{Baroncini-Bonder, Baroncini-Bonder-Spedaletti,Bucur-Trebeschi, Bucur-Zolesio}.

In the last years the interest in problems involving behaviors more general than powers has considerably increased due to the application in describing models with non-standard growth. See for instance \cite{Allaire, Bucur-Butazzo, DSSSS, HenrotPierre, Pironneau, Sokol} and references therein. Moreover, recently this kind of behaviors were treated in a non-local framework. See \cite{FBS,S}.

Given a so-called Young function $\Phi\colon \R_+\to \R$ such that $\Phi'=\phi$ we consider the well-known $\phi-$Laplacian $\Delta_\phi$ defined as
$$
\Delta_\phi:=\diver \left (\frac{\phi(|\nabla u|)}{|\nabla u|}\cdot\nabla u\right )
$$
whose natural space to work with  is the Orlicz-Sobolev one $W^{1, \Phi}(\Omega)$. It is worth to mention that when $\Phi(t)=t^p/p$, with $p>1$,  the aforementioned operator becomes the usual $p-$Laplacian, however, in its full generality, $\Delta_\phi$ is not homogeneous. This fact, in contrast with the case of powers,  brings on many technical difficulties  to be overcome in this manuscript.

The bibliography on Orlicz spaces in vast. We recommend, for instance, the books \cite{Orlicz-Birnbaum, KR, Fucik-John-Kufner, RR} for further information and additional topics in this theory.

Given a sequence of open sets $\{\Omega_k\}_{k\in\N}$ contained in a fix open and bounded $D\subset \R^n$, and given  a fix function $f$ belonging to the dual space of $W^{1,\Phi}_0(D)$, the main goal of this manuscript is to study under which conditions we can claim that solutions $u_{\Omega_k,f}$ (in the weak sense) of the problem
\begin{align}\label{p1}
\begin{cases}
-\Delta_\phi u_{\Omega_k,f} = f &\quad \text{ in } \Omega_k\\
u_{\Omega_k,f}=0 &\quad \text{ on } \partial \Omega_k
\end{cases}
\end{align}
converge (weakly in $W^{1,\Phi}_0(D)$) to the weak solution $u_{\Omega,f}$ of the limit problem
\begin{align}\label{p2}
\begin{cases}
-\Delta_\phi u_{\Omega,f} = f &\quad \text{ in } \Omega\\
u_{\Omega,f}=0 &\quad \text{ on } \partial \Omega,
\end{cases}
\end{align}
where $W^{1,\Phi}_0(\Omega)$ is the closure of the smooth compactly supported functions in the norm of $W^{1,\Phi}(\Omega)$. This notion of convergence is usually referred in the literature as $\gamma-$convergence of the sequence $\{\Omega_k\}_{k\in\N}$ to the limit set $\Omega$, and we denote it  $\Omega_k\tog \Omega$ as $k\to\infty$.  

Our main result establishes precise condition to guarantee  the $\gamma-$converges of the sequence $\{\Omega_k\}_{k\in\N}\subset D$. More precisely, in Theorem \ref{main} we prove that, if the seqeunce of domains fulfills
\begin{itemize}
\item[(a)] $\Omega_k \to \Omega$ \quad in the Hausdorff complementary topology (see definition \ref{Hausdorff})
\item[(b)] the $\Phi-$capacity of $\Omega_k\setminus \Omega$ relative to $D$ vanishes as $k\to\infty$ (see section \ref{capacidad}),
\end{itemize}
then we can deduce that $\Omega_k\tog \Omega$ as $k\to\infty$. We additionally prove that, in fact, the convergence of $u_{\Omega_k,f}$ to $u_{\Omega,f}$ is strong in $W^{1,\Phi}_0(D)$ (see Remark \ref{rem.fuerte}).

In general, to check (b) for a given sequence it is not a straightforward task. In Theorem \ref{main.2} we prove that, if (a) holds, and moreover, $D$ has Lipschitz boundary and the Young function $\Phi$ additionally fulfills the integrability condition
\begin{equation} \label{cond.intro}
\int_1^\infty \frac{\Phi^{-1}(t)}{t^{1+1/n}}\, dt <\infty,
\end{equation}
then $\Omega_k \tog \Omega$ as $k\to \infty$. It can be checked that  condition 	\eqref{cond.intro} is fulfilled for instance when $p^-+1>n$.

We remark that Theorem \ref{main.2} can be seen as a \emph{weak \v{S}ver\'{a}k result} in the sense that we  cannot guarantee that condition \eqref{cond.intro} is the optimal one in that theorem. On the counterpart of the case of powers dealt  in \cite{Bucur-Trebeschi, Sverak}, we conjecture that condition \eqref{cond.intro} can be improved (up to $p^->n$) if additionally is required that the number of connected components of each $\Omega_k$ is finite. Such a  result would have deep implications, and its proof seems to be highly nontrivial. In section \ref{sec.open} we formulate some questions on this subject which remain open.

We conclude this introduction by mentioning that the results exposed in this paper can be extended to a more general family of operators, including for instance non-local ones. However, for the sake of simplicity in the proofs and notations we have chosen to deal with the prototypical case of the $\phi-$Laplacian.

\subsection*{Organization of the paper}

The paper is organized as follows. In section \ref{sec.preliminar} we recall some definitions and properties on Young functions and Orlicz-Sobolev spaces; moreover, for the sake of completeness we prove some useful results on existence and uniqueness of solutions of problems involving nonstandard growth operators; we conclude that section by introducing the notion of $\Phi-$capacity and some of their main properties. In section \ref{continuidad} we prove our main results. Finally, in Section \ref{sec.rem} we deliver some further extensions and generalizations.

\section{Preliminaries} \label{sec.preliminar}

\subsection{Young functions}
We say that a function $\Phi\colon\R_+\to \R_+$ belongs to the \emph{Young class} if it admits the integral formulation $\Phi(t)=\int_0^t \phi(\tau)\,d\tau$, where the right continuous function $\phi$ defined on $[0,\infty)$ has the following properties:
\begin{align*} 
&\phi(0)=0, \quad \phi(t)>0 \text{ for } t>0 \label{g0} \tag{$\phi_1$}, \\
&\phi \text{ is nondecreasing on } (0,\infty) \label{g2} \tag{$\phi_2$}, \\
&\lim_{t\to\infty}\phi(t)=\infty  \label{g3} \tag{$\phi_3$} .
\end{align*}
From these properties it is easy to see that a Young function $\Phi$ is continuous, nonnegative, strictly increasing and convex on $[0,\infty)$. Without loss of generality $\Phi$ can be normalized such that $\Phi(1)=1$.

The \emph{complementary Young function} $\Phi^*$ of a Young function $\Phi$ is defined as
$$
\Phi^*(t):=\sup\{tw -\Phi(w): w>0\}.
$$ 
From this definition the following Young-type inequality holds
\begin{equation} \label{Young}
ab\leq \Phi(a)+\Phi^*(b)\qquad \text{for all }a,b\geq 0.
\end{equation}
Moreover, it is not hard to see that $\Phi^*$ can be written in terms of the inverse of $\phi$ as
\begin{equation} \label{xxxx}
\Phi^*(t)=\int_0^t \phi^{-1}(\tau)\,d\tau,
\end{equation}
see \cite[Theorem 2.6.8]{RR}.

The following growth condition on the Young function $\Phi$ will be assumed as
\begin{equation} \label{cond} \tag{L}
1<p^-\leq \frac{t\phi'(t)}{\phi(t)} \leq p^+<\infty \quad \forall t>0
\end{equation}
where $p^\pm$ are fixed numbers.

Observe that from \eqref{cond} it follows that
\begin{equation} \label{condL} \tag{L'}
2<p^- +1 \leq \frac{t\phi(t)}{\Phi(t)} \leq p^+ +1<\infty \quad \forall t>0.
\end{equation}

Roughly speaking,   condition \eqref{cond} tells us that $\Phi$ remains between two power functions.

\begin{remark} \label{condic}
Observe that, from \eqref{cond} and \eqref{condL} we have the relation
\begin{equation} 
t^2 \phi'(t)\geq p^- (p^- +1) \Phi(t) \quad \text{ for all } t\geq 0.
\end{equation}

\end{remark}
The following  properties are well-known in the theory of Young function. We refer, for instance, to the books \cite{KR,Fucik-John-Kufner,RR} for an introduction to Young functions and Orlicz spaces, and the proof of these results. See also \cite{FBPLS}.

We remark that condition \eqref{cond} implies a monotonicity property of the Young function $\Phi$, that is, for any $a,b\in\R^n$ it holds that
$$
\left(\frac{\phi(|a|)}{|a|}a-\frac{\phi(|b|)}{|b|}b \right)\cdot (a-b) \geq 0.
$$
However, for our purposes we need a refined version of the monotonicity, i.e., a lower bound in terms of $\Phi$  of the type
$$
\left(\frac{\phi(|a|)}{|a|}a-\frac{\phi(|b|)}{|b|}b \right)\cdot (a-b) \geq C \Phi(|a-b|)
$$ 
for some constant $C$ depending on $\Phi$. That is the content of Lemma \ref{lema4.1}. To obtain that, as it happens in the case of powers,   $\Phi$ needs to satisfy an additional assumption, more precisely, 
\begin{equation} \label{condC} \tag{C}
	\Phi'(t) \text{ is a convex function for all }t\geq 0.
\end{equation}
 
We recall some useful properties of Young functions. 
\begin{lema} \label{propiedades}
Let $\Phi$ be a Young function satisfying \eqref{cond} and $a,b\geq 0$. Then 
\begin{align*}
  &\min\{ a^{p^-+1}, a^{p^++1}\} \Phi(b) \leq \Phi(ab)\leq   \max\{a^{p^-+1},a^{p^++1}\} \Phi(b),\tag{$\Phi_1$}\label{G1}\\
  &\Phi(a+b)\leq \C (\Phi(a)+\Phi(b)) \quad \text{with } \C:=  2^{p^+ +1},\tag{$\Phi_2$}\label{G2}\\
	&\Phi \text{ is Lipschitz continuous}. \tag{$\Phi_3$}\label{G3}
 \end{align*}
\end{lema}
Condition \eqref{G2} is known as the \emph{$\Delta_2$ condition} or \emph{doubling condition} and, as it is showed in \cite[Theorem 3.4.4]{Fucik-John-Kufner}, it is equivalent to the right hand side inequality in \eqref{cond}.

It is easy to see that   condition \eqref{cond} implies that
\begin{equation}
(p^++1)'\leq \frac{t(\Phi^*)'(t)}{\Phi^*(t)} \leq (p^-+1)' \quad \forall t>0,\tag{$\Phi^*_3$}
\end{equation}
from where it follows that $\Phi^*$ also  satisfies the $\Delta_2$ condition.
\begin{lema} 
Let $\Phi$ be a Young function satisfying \eqref{cond} and $a,b\geq 0$. Then 
\begin{equation} \label{g.s.1} \tag{$G^*_1$}
  \min\{a^{(p^-+1)' },a^{(p^++1)' }\} \Phi^*(b)\leq \Phi^*(ab)\leq    \max\{a^{(p^-+1)'}, a^{(p^++1)'}\} \Phi^*(b),
\end{equation}
where $(p^\pm+1)' = \frac{p^\pm+1}{p^\pm}$.	
\end{lema}

Since $\phi^{-1}$ is increasing, from \eqref{xxxx} and \eqref{cond} it is immediate the following relation.
\begin{lema} \label{lemita}
Let $\Phi$ be an Young function satisfying \eqref{cond} such that $\phi=\Phi'$ and denote by  $\Phi^*$ its complementary function. Then
$$
\Phi^*(\phi(t)) \leq (p^+ +1) \Phi(t)
$$
holds for any $t\geq 0$.
\end{lema}

From \cite[Lemma 3.12.3]{Fucik-John-Kufner} we get the following density result.
\begin{lema} \label{densidad}
Let $\Omega\subset \R^n$ be open and bounded. If $\Phi$ is such that the left hand side of \eqref{cond}  holds then $L^\infty(\Omega)$ is dense in $L^{\Phi^*}(\Omega)$.
\end{lema}

The following result will be key in our arguments.
\begin{lema} \label{Phi.tilde}
Given $\Phi$ satisfying \eqref{cond}, then the function
$$
\tilde \Phi(t):=\Phi(\sqrt{t}), \qquad t\geq 0
$$
is convex and
\begin{equation} \label{Phi.tilde.prop}
\min\{ a^{ \frac{p^-+1}{2}}, a^\frac{p^-+1}{2}\} \tilde\Phi(b) \leq \tilde \Phi(ab)\leq   \max\{a^\frac{p^-+1}{2},a^\frac{p^-+1}{2}\} \tilde\Phi(b).
\end{equation}
\end{lema}
\begin{proof}
A direct computation gives that
$$
\tilde \Phi''(t)=\frac{\phi'(\sqrt{t})}{4t} - \frac{\phi(\sqrt{t})}{4t^\frac32}, \qquad t\geq 0.
$$
Observe that $\tilde \Phi''(t)\geq 0$ if and only if $\sqrt{t} \phi'(\sqrt{t}) \geq \phi(\sqrt{t})$, which is true since we are assuming \eqref{cond}. Moreover, \eqref{Phi.tilde.prop} is immediate from \eqref{G1}.
\end{proof}

Another useful result regarding strong convergence is the following.
\begin{prop} [\cite{RR}, Theorem 12] \label{teo.rr}
Given $\Omega\subset\R^n$ be an open domain. Let $\{u_k\}_{k\in\N}$ be a sequence in $L^\Phi(\Omega)$ and $u\in L^\Phi(\Omega)$. If $\Phi^*$ satisfies the $\Delta_2$ condition,
 $$
\int_{\Omega} \Phi(|u_k|)\, dx\to \int_{\Omega} \Phi(|u|)\, dx,
 $$  
 and  $u_k\to u$ a.e. in $\Omega$, then $u_k\to u$ in $L^\Phi(\Omega)$.
\end{prop}

\begin{exam}\label{ej.orlicz}
The family of Young functions includes the following examples.
\begin{enumerate}
\item \emph{Powers}.
If $\phi(t)=t^{p-1}$, $p>1$ then $\Phi(t)=\frac{t^p}{p}$, and $p^\pm = p$.

\item \emph{Powers$\times$logarithms}. Given $b,c>0$ if $\phi(t)=t\log(b+ct)$ then
$$
\Phi(t)=\frac{1}{4c^2}\left(ct(2b-ct)-2(b^2 -c^2 t^2) \log(b+ct) \right)
$$
and $p^-=1$, $p^+=2$. In general, if $a,b,c>0$ and $\phi(t)=t^a\log(b+ct)$ then
$$
\Phi(t)=\frac{t^{1+a}}{(1+a)^2}\left( {}_2 F_1(1+a,1,2+a,-\tfrac{ct}{b})  + (1+a)\log(b+ct)-1 \right)
$$
with $p^-=a$, $p^+=1+a$, where ${}_2 F_1$ is a hyper-geometric function.

\item \emph{Different powers behavior}. An important example is the family of functions $\Phi$ allowing different power behavior near $0$ and infinity. The function $\Phi$ can be considered  such that 
$$
\phi\in C^1([0,\infty)), \quad \phi(t)=c_1 t^{a_1} \text{ for } t\leq t_0 \quad \text{ and } \quad  \phi(t)=c_2 t^{a_2}+d \text{ for } t\geq t_0.
$$
In this case $p^-=\min\{a_1,a_2\}$ and $p^+=\max\{a_1,a_2\}$.

\item \emph{Linear combinations}. If $\phi_1$ and $\phi_2$ satisfy \eqref{cond} then $a_1 \phi_1 + a_2 \phi_2$ also satisfies \eqref{cond} when $a_1,a_2\geq 0$.

\item  \emph{Products}. If $\phi_1$ and $\phi_2$ satisfy \eqref{cond} with constants $p^\pm_k$, $i=1,2$, then $\phi_1\phi_2$ also satisfies \eqref{cond} with constants $p^-=p^-_1+p^-_2 $ and $p^+=p^+_1+p^+_2$.

\item \emph{Compositions}. If $\phi_1$ and $\phi_2$ satisfy \eqref{cond} with constants $p_k^\pm$, $i=1,2$, then $\phi_1\circ \phi_2$ also satisfies \eqref{cond} with constants $p^-=p^-_1p^-_2$ and $p^+=p^+_1p^+_2$.
\end{enumerate}
\end{exam}

\subsection{Orlicz-Sobolev spaces}
Given a Young function $\Phi$ and $\Omega\subset \R^n$ an open domain, we introduce the well-known Orlicz and Orlicz-Sobolev spaces defined as
\begin{align}\label{OS} 
\begin{split}
L^\Phi(\Omega)&:=\{u\colon \R^n \to \R \text{ measurable s.t.} \int_\Omega \Phi(|u|)\,dx <\infty\}\\
W^{1,\Phi}(\Omega)&:=\left\{u \in L^\Phi(\Omega) \text{ s.t. } |\nabla u|\in L^\Phi(\Omega) \right \},
\end{split}
\end{align}
where the partial derivatives are understood in the distributional sense. These spaces are separable  Banach spaces endowed with the Luxemburg norms
\begin{equation}\label{Luxemburg}
	\|u\|_{L^\Phi(\Omega)}=\inf \left \{\lambda >0\colon \int_\Omega \Phi\left (\frac{|u|}{\lambda}\right)\, dx\leq 1\right \}.
\end{equation}
and
\begin{equation} \label{norma}
\|u\|_{W^{1,\Phi}(\Omega)}= \|u\|_{L^\Phi(\Omega)} + \|\nabla u\|_{L^\Phi(\Omega)},
\end{equation}
respectively.

Moreover, it is well-known that these spaces are reflexive when $\Phi$ fulfills   condition \eqref{cond}. See \cite{Fucik-John-Kufner}.

The natural space to deal with Dirichlet boundary condition is the space  $W_0^{1,\Phi}(\Omega)$ which is defined as
$$
W^{1,\Phi}_0(\Omega)=\overline{C_c^\infty(\Omega)}.
$$
where the closure is taken on the Orlicz-Sobolev norm. When $\partial\Omega$ is smooth enough (for instance $C^1$) and $\Phi$ satisfies \eqref{cond}, functions in $W^{1,\Phi}_0(\Omega)$ can be identified with those vanishing on $\partial\Omega$.

The following Poincar\'e type inequality holds in $W^{1,\Phi}_0(\Omega)$.
\begin{lema}\label{poincare}
Let $\Omega\subset\R^n$ be open and bounded in one direction. Let $\Phi$ be a Young function satisfying \eqref{cond}. Then there exists a positive constant $C_P=C_P(\Omega,\Phi)$ such that
\begin{equation} \label{des.poinc}
\int_\Omega\Phi
(|u|)\,dx\leq C_P\int_\Omega\Phi(|\nabla u|)\, dx,
\end{equation}
for all $u\in W^{1,\Phi}_0(\Omega)$.
\end{lema}

\begin{proof}
Let $u\in C_c^\infty(\Omega)$. We write $u(x)=u(x_1,x')$ where $(x_1,x')\in  \R\times \R^{n-1}$, since $u$ has   compact support there exists an interval $(a,b)$ such that
$$
|u(x)|\leq (b-a)\pint_{(a,b)} |\nabla u(t,x')|\,dt.
$$
By using \eqref{G1}, \eqref{G3}, Jensen inequality, and by integrating with respect to $x'$ we get \eqref{des.poinc} with a constant $C_P$ depending on $|b-a|$ and $p^\pm$. Finally the result follows by a standard density argument.
\end{proof}
\begin{remark}\label{ctePoincare}
	We observe that if $\Omega\subset \R^n$ is bounded, the constant $C_P$ in the previous inequality only depends of the parameters $p^\pm$ and the diameter of $\Omega$.
\end{remark}

\begin{remark}
Lemma \ref{poincare} implies that $\|\nabla u\|_\Phi$ is an equivalent norm in $W^{1,\Phi}_0(\Omega)$. Indeed,  if $\bar{C}_P=\max\{1,C_P\}$ where $C_P$ is given in Lemma \ref{poincare}, we get for $u\in W^{1,\Phi}_0(\Omega)$
$$
\int_\Omega \Phi\left (\frac{|u|}{\bar{C}_P \|\nabla u\|_\Phi}\right )\, dx\leq C_P\int_\Omega\Phi \left ( \frac{|\nabla u|}{\bar{C}_P \|\nabla u\|_\Phi }\right )\,dx\leq \int_\Omega\Phi \left ( \frac{|\nabla u|}{\|\nabla u\|_\Phi}\right )\,dx\leq 1,
$$ 
and therefore $\|u\|_\Phi \leq \bar{C}_P \|\nabla u\|_\Phi$, implying that $
 \|\nabla u\|_\Phi \leq \|u\|_{1,\Phi} \leq (\bar{C}_P+1)\|\nabla u\|_\Phi$.
\end{remark}
By \cite[Theorem 7.4.4]{Fucik-John-Kufner} and the observation in \cite[Example 6.3]{DSSSS} we have the following compactness result for Orlicz-Sobolev spaces.

\begin{teo} \label{compacidad}
Let $\Omega\subset\R^n$ be open and bounded  Lipschitz domain  and let $\Phi$ be a Young function.
\begin{itemize}
\item[(a)] The immersion $W^{1,\Phi}(\Omega) \subset \subset L^\Phi(\Omega)$ is compact whenever 
\begin{equation}\label{condint1}
\int_1^\infty \frac{\Phi^{-1}(s)}{s^{1+1/n}}\,ds=\infty.
\end{equation}

\item[(b)]  The immersion $W^{1,\Phi}(\Omega)\subset C^{0,\sigma(t)}(\bar{\Omega})$ is  compact  whenever
\begin{equation}\label{condint2}
\int_1^\infty \frac{\Phi^{-1}(s)}{s^{1+1/n}}\,ds<\infty,
\end{equation}
where $\sigma(t)=\int_{t^{-n}}^\infty \frac{\Phi^{-1}(s)}{s^{1+1/n}}\,ds.$
\end{itemize}
%On the other hand if the Young function $\Phi$ satisfies the condition
%\begin{equation}\label{condint2}
%\int_1^\infty \frac{\Phi^{-1}(s)}{s^{1+1/n}}\,ds<\infty,
%\end{equation}
%then the immersion $W^{1,\Phi}(\Omega)\subset C^{0,\sigma(t)}(\bar{\Omega})$ is compact, for every function $\sigma$ such that
%$$
%\sigma \prec\prec \sigma_0 \text{ with }\sigma_0(t)=\int_{t^{-n}}^\infty \frac{\Phi^{-1}(s)}{s^{1+1/n}}\,ds.
%$$
\end{teo}

The following strong maximum principle holds for Orlicz-Sobolev functions (see Theorem 1 in  \cite{Montenegro}).
\begin{lema}[Strong maximum principle]  \label{princ.max.fuerte}
Let $\Omega\subset\R^n$ be an open domain, and $\Phi$ a Young function satisfying \eqref{cond}. Given a nonnegative $f\in L^{\Phi^*}(\Omega)$, if the weak solution $u_{\Omega,f}\in W^{1,\Phi}_0(\Omega)$ of $-\Delta_\phi=f$ in $\Omega$ satisfies $u_{\Omega,f}\geq 0$ in $\Omega$, then either $u_{\Omega,f}\equiv0$ in $\Omega$, or $u_{\Omega,f}>0$ in $\Omega$.
\end{lema}
As a corollary of the above lemma we obtain the following results.
\begin{lema}[Weak maximum principle] \label{weak.mp}
Let $\Omega\subset\R^n$ be an open domain, and $\Phi$ a Young function satisfying \eqref{cond}. Then, if $f\in L^{\Phi^*}(\Omega)$ is nonnegative in $\Omega$, then the weak solution $u_{\Omega,f}\in W^{1,\Phi}_0(\Omega)$ of $-\Delta_\phi=f$ in $\Omega$ satisfies $u\geq 0$ in $\Omega$.
\end{lema}

\begin{lema}[Comparison] \label{comparison}
Let $u,v\in W^{1,\Phi}(\Omega)$ be such that
\begin{itemize}
\item[(i)] $v\geq u$ in $\partial \Omega$ 
\item[(ii)]  for any nonnegative $\psi\in C_c^\infty(\Omega)$
$$\int_\Omega \frac{\phi(|\nabla v|)}{|\nabla v|}\nabla v \cdot \nabla \psi\,dx \geq \int_\Omega \frac{\phi(|\nabla u|)}{|\nabla u|}\nabla u \cdot \nabla \psi\,dx.
$$
\end{itemize}
Then $v\geq u$ in $\overline\Omega$.
\end{lema}

\subsection{Existence of solutions}
We prove some basic result on existence of solutions of the Dirichlet problem for the so-called  $\phi-$Laplacian, which, given a Young function $\Phi$ is defined as 
$$
\Delta_\phi(u):=\diver\left( \frac{\phi(|\nabla u|)}{|\nabla u|} \nabla u\right)
$$ 
where $\phi=\Phi'$. More precisely, given $f\in L^{\Phi^*}(\Omega)$ and an open set $\Omega\subset \R^n$, we are interested in obtaining existence of solutions of the following Dirichlet problem
\begin{equation}\label{P}
\begin{cases}
-\Delta_{\phi} u= f & \text{in } \Omega,\\
u=0 & \text{in } \partial \Omega.
\end{cases}
\end{equation}

\begin{defi}\label{sd}
A function $u=u_{\Omega,f}\in W^{1,\Phi}_0(\Omega)$ is a \emph{weak solution} of  \eqref{P} if
\begin{equation}\label{debil}
\int_\Omega \frac{\phi(|\nabla u|)}{|\nabla u|}\nabla u \cdot \nabla v\, dx=\int_\Omega f v\, dx \qquad \forall v\in W^{1,\Phi}_0(\Omega).
\end{equation}

\end{defi}

In order to study  solutions of \eqref{P} we define the functional $J_\Phi\colon W^{1,\Phi}_0(\Omega)\to \bar{\R}$ 
\begin{equation}\label{J}
J_{\Phi}(u):=\int_\Omega \Phi(|\nabla u|)\, dx- \int_\Omega fu\, dx,
\end{equation}
and the corresponding minimization problem
\begin{equation}\label{min}
\inf_{u\in W^{1,\Phi}_0(\Omega)} J_\Phi(u). 
\end{equation}
%Observe that Lemma \ref{poincare} implies that the infimum in \eqref{min} is achieved since $J_\Phi$ is lower bounded. Indeed, given   an arbitrary function $u\in W^{1,\Phi}_0(\Omega)$ and  $0<\varepsilon <1$, from  Young's inequality we get 
%\begin{align*}
%	\int_{\Omega}fu\, dx  \leq \int_{\Omega} \Phi(\varepsilon |u| ) \,dx + \int_{\Omega}\Phi^*\left( \left|\frac{f}{\varepsilon} \right|\right) dx \leq  \varepsilon \int_{\Omega} \Phi(|u| ) \,dx + \int_{\Omega}\Phi^*\left( \left|\frac{f}{\varepsilon} \right|\right) dx.
%\end{align*}
%Then, from the definition of $J_{\Phi}(u)$ and  Lemma \ref{poincare}  we have
%\begin{align*}
%	J_{\Phi}(u) &\geq (1-\varepsilon C) \int_{\Omega} \Phi(|\nabla u| )\, dx - \int_{\Omega}\Phi^*\left( \left|\frac{f}{\varepsilon} \right|\right)\, dx.
%\end{align*} 
%Finally, taking $\varepsilon<\min\{1, C^{-1}\}$ we arrive at
%$$
%J_\Phi(u)\geq  - \int_{\Omega}\Phi^*\left( \left|\frac{f}{\varepsilon} \right|\right)\, dx.
%$$

The following two results show an implication  between solutions in the weak sense  and minimizers of the functional \eqref{min}.

\begin{prop} \label{prop.c1}
The functional $\mathcal{F}\colon W^{1,\Phi}_0(\Omega)\to \R$ defined as $\mathcal{F}(u)=\int_\Omega \Phi(|\nabla u|)\,dx$ is class $C^1$ and its Fr\'echet derivative $\mathcal{F}'\colon W^{1,\Phi}_0(\Omega)\to  (W^{1,\Phi}_0(\Omega))'$ satisfies that
$$
\langle \mathcal{F}'(v),v \rangle = \int_\Omega \frac{\phi(|\nabla u|)}{|\nabla u|} \nabla u\cdot \nabla v\,dx.
$$
\end{prop}
\begin{proof}
Given $u,v \in W^{1,\Phi}_0(\Omega)$ and $t>0$ 
$$
\frac{f(u+tv)-g(u)}{t} = \int_\Omega \frac1t \int_{|\nabla u|}^{|\nabla u + t \nabla v|} \phi(s)\,ds \,dx
$$
and, as $t\to 0$, $|\nabla u + t \nabla v| \to |\nabla v|$ as $t\to 0$ in $L_\Phi$ and hence in $L^1$, thus, up to a subsequence we may assume convergence almost everywhere. Moreover, since $g$ is increasing, for $t$ small we get
$$
\left| \frac1t \int_{|\nabla u|}^{|\nabla u + t \nabla v|} \phi(s)\,ds \right| \leq \phi(|\nabla u| + |\nabla v|)|\nabla v|.
$$
We claim that $\phi(|\nabla w|)\in L^{\Phi^*}(\Omega)$ for all $w\in W^{1,\Phi}_0(\Omega)$. Indeed, using \eqref{xxxx} and the fact that $\phi^{-1}$ is increasing we obtain that
\begin{align*}
\int_\Omega \Phi^*(|\nabla w|) &= \int_\Omega \left(\int_0^{\phi(|\nabla w|)} \phi^{-1}(s)\,ds  \right) \,dx\\
&\leq \int_\Omega \phi^{-1}( \phi(|\nabla w|) )\phi(|\nabla w|)\,dx\\
&\leq (p^++1) \int_\Omega \Phi(|\nabla w|)\,dx.
\end{align*}
Then, since $u,v\in W^{1,\Phi}_0(\Omega)$ we get that $\phi(|\nabla u+\nabla v|)|\nabla v| \in L^1(\Omega)$. Thus, by dominated convergence theorem, 
$$
\langle \mathcal{F}'(u),v \rangle = \lim_{t\to 0} \frac{\mathcal{F}(u+tv)-\mathcal{F}(v)}{t}= \frac{d}{dt}\mathcal{F}(u+tv)\big|_{t=0} = \int_\Omega \frac{\phi(|\nabla u|)}{|\nabla u|} \nabla u\cdot \nabla v\,dx.
$$
Let us see now that $\mathcal{F}'$ is continuous. Let $\{u_k\}_{j\in\N}\subset W^{1,\Phi}_0(\Omega)$ be such that $u_k\to u$ and observe that
\begin{align*}
|\langle \mathcal{F}'(u_k) - \mathcal{F}(u),v\rangle| = \left| \int_\Omega \left(\frac{\phi(\nabla u)}{|\nabla u|} \nabla u - \frac{\phi(\nabla u_k)}{|\nabla u_k|} \nabla u_k \right)\nabla v\,dx \right|,
\end{align*}
then by using Egoroff's Theorem there exists a positive sequence $\delta_k\to 0$ such that
\begin{align*}
\sup_{\|v\|_{1,\Phi}\leq 1}  \int_\Omega &\left(\frac{\phi(\nabla u)}{|\nabla u|} \nabla u - \frac{\phi(\nabla u_k)}{|\nabla u_k|} \nabla u_k \right)\nabla v\,dx\\
 &\leq \left\| \frac{\phi(\nabla u)}{|\nabla u|} \nabla u - \frac{\phi(\nabla u_k)}{|\nabla u_k|} \nabla u_k \right\|_{L^{\Phi^*}(\Omega)} + \delta_k
\end{align*}
where we have used H\"older's  inequality for Orlicz spaces. Now, since $\Phi^*$ satisfies \eqref{g.s.1}, by Proposition \ref{teo.rr} we get
$$
\| \phi(\nabla u)- \phi(\nabla u_k)\|_{L^{\Phi^*}(\Omega)}\to 0
$$
and therefore $\| \mathcal{F}'(u_k) - \mathcal{F}'(u)\|_{L^{\Phi^*}(\Omega)}\to 0$.
\end{proof}

\begin{teo} \label{teo.existencia}
Let $\Omega\subset\R^n$ be an open and  bounded set, and let $\Phi$ be a Young function satisfying \eqref{cond}, then there a exists solution for the minimization problem
\begin{equation} \label{minim}
\inf_{W^{1,\Phi}_0(\Omega)}J_\Phi.
\end{equation}
Moreover, if the Young function $\Phi$ is strictly convex, then the solution is unique.
\end{teo}
\begin{proof}
Let $\{ u_k \}_{k\in\N} \in W^{1,\Phi}_0(\Omega)$ be a minimizing sequence, that is 
\begin{equation*}
	\lim_{k\to \infty} J_{\Phi}(u_k)= \inf_{W^{1,\Phi}_0(\Omega)} J_{\Phi}(u),
\end{equation*}
then there exists a positive constant $C$ such that
\begin{equation*}
\int_{\Omega} \Phi(|\nabla u_k| )\, dx \leq C+\int_{\Omega} f u_k \, dx \quad  \mbox{  for all } k \in \N.
\end{equation*}
By using Young's inequality for Young functions (see \cite{Fucik-John-Kufner}), for $0<\varepsilon<1$ together with Lemma \ref{poincare} we get
\begin{align*}
	\int_{\Omega} \Phi(|\nabla u_k| )\, dx &\leq C + \int_{\Omega} f u_k \,dx\\
	&\leq C + \int_{\Omega} \Phi(\varepsilon u_k) \, dx + \int_{\Omega} \Phi^* \left( \frac{f}{\varepsilon}  \right) \,dx\\
	&\leq C + \varepsilon C\int_{\Omega} \Phi(|\nabla u_k|) \, dx + \int_{\Omega} \Phi^* \left( \frac{f}{\varepsilon}  \right) \,dx.
\end{align*}
%from where we obtain that
%\begin{equation*} 
%	(1-\varepsilon C)\int_{\Omega} \Phi (|\nabla u_k|) \, dx \leq C + \int_{\Omega} \Psi \left( \frac{f}{\varepsilon}  \right) \,dx.
%\end{equation*}
Choosing $0<\varepsilon < \min \{ C^{-1},1 \} $, we obtain
\begin{equation*}
\int_{\Omega} \Phi (|\nabla u_k|) \, dx \leq \frac{ C + \int_{\Omega} \Phi^* \left( \frac{f}{\varepsilon}  \right) \,dx}{(1-\varepsilon C)}:=M,
\end{equation*}
and then by defining $\bar{M}= \max\{1,M\}\geq 1$ it follows that
\begin{equation*}
	\int_{\Omega} \Phi \left( \frac{|\nabla u_k|}{\bar{M}}\right) \, dx \leq \frac{1}{\bar{M}} \int_{\Omega} \Phi (|\nabla u_k|) \, dx \leq 1,
\end{equation*}
and by the definition of the Luxemburg norm we get that $\|\nabla u_k\|_{L^\Phi(\Omega)}$ is uniformly bounded by $\bar M$ for any $k\in\N$. Therefore, from the reflexivity of $W^{1,\Phi}_0(\Omega)$, 
up to  a subsequence there exists $u\in W^{1,\Phi}_0(\Omega)$ such that 
\begin{align} \label{converg}
\begin{split}
u_k\rightharpoonup u \quad &\text{ weakly in } W^{1,\Phi}(\Omega). 
\end{split}
\end{align}
Observe that by definition we have that $\inf\{ J_\Phi(u)\colon u \in W^{1,\Phi}_0(\Omega) \}\leq J_\Phi(u).$

On the other hand, by Theorem 2.2.8 in \cite{DHH} we have that that the modular is lower semicontinuous with respect to weak convergence, then from \eqref{converg} we get
\begin{align*}
\inf_{W^{1,\Phi}_0(\Omega)}J_\Phi& \geq \liminf_{k\to\infty}\int_\Omega\Phi(|\nabla u_k|)\,dx-\lim_{k\to\infty}\int_\Omega f u_k \,dx\\
&\geq  \int_\Omega\Phi(|\nabla u|)\,dx-\int_\Omega f u\,dx
= J_\Phi(u),
\end{align*}
from where $u$ minimizes \eqref{minim}.

Finally, if $\Phi$ is strictly convex and $u_1,u_2$ in $W^{1,\Phi}_0(\Omega)$ are two different solutions we have that
$$
J_\Phi\left (\frac{u_1+u_2}{2}\right)<\frac{1}{2}J_\Phi(u_1)+\frac{1}{2}J_\Phi(u_2)=\inf_{W^{1,\Phi}_0(\Omega)}J_\Phi,
$$
which is absurd. The proof is now complete.
\end{proof}

\begin{teo}
If $u\in W^{1,\Phi}_0(\Omega)$ minimizes  \eqref{minim} then $u$ solves \eqref{debil}.
\end{teo}
\begin{proof}
Let $u\in W^{1,\Phi}_0(\Omega)$ be a solution for  problem \eqref{minim}, then in  light of  Proposition \ref{prop.c1}, for $v\in W^{1,\Phi}_0(\Omega)$ we have
$$
\frac{d}{dt}J_\Phi(u+tv)\Big|_{t=0}=0,
$$
from where we get
\eqref{debil}.
\end{proof}

\subsection{Orlicz-Sobolev capacity}\label{capacidad}
In this subsection we introduce some definitions and basic facts on capacities in Orlicz-Sobolev spaces. For further information we refer for instance to \cite{DHH2,Ohno-Tetsu}. Thought this section $\Phi$ stands for a Young function. 
 
Given $E\subset \R^n$ we consider the set
$$
S_\Phi (E) =\{u \in W^{1,\Phi}(\R^n)\colon u\geq 1 \text{ in an open set containing }E \}.
$$

The \emph{Sobolev capacity} is defined by
$$
\cp(E)= \inf_{u\in S_{\Phi}(E)}  \rho_\Phi(u),
$$
where, for $u\in W^{1,\Phi}(\R^n)$ we denote
$$
\rho_\Phi(u) = \int_{\R^n}  \Phi(|u|) \,dx +   \int_{\R^n}  \Phi(|\nabla u|) \,dx.
$$
In case that $S_{\Phi} (E)=\emptyset$, we set $\cp(E)=\infty$.

It is well-known that when dealing with pointwise properties of Sobolev functions, the concept of almost everywhere needs to be changed to quasi everywhere.  

We say that a property holds \emph{$\Phi-q.e$ (quasi everywhere)} in $\Omega$, if it holds except of a set $F\subset \Omega$ such that $\cp(F)=0$.

 A function u is \emph{$\Phi-$quasicontinuous} on $\Omega$ if, for any $\varepsilon >0$, there is an open set $E$ such that $\cp(E)<\varepsilon$ and $u|_{\Omega\setminus E}$ is continuous.

The following lemma will be useful to our proof in Section \ref{continuidad}.
\begin{lema}\cite{DHH}[Section 11.1.11] \label{prop.c}
Let $v_k\to v$ in $W^{1,\Phi}_0(D)$. Then, up to a subsequence, $\tilde v_k \to \tilde v$ $\Phi-$q.e.
\end{lema}
We give now a characterization of $W^{1,\Phi}_0(\Omega)$ as the restriction of quasicontinuous functions vanishing quasi everywhere on $\partial\Omega$. 

\begin{prop}\cite[Theorem A.13]{Ohno-Tetsu}  \label{eqiv.espacios}
Let $D\subset \R^n$ be open and   $\Omega\subset D$. Then the following is equivalent
\begin{itemize}
\item[(i)] $u\in W^{1,\Phi}_0(\Omega)$
\item[(ii)] there exists a $\Phi-$quasicontinuous function $\tilde u\in W^{1,\Phi}(D)$ such that  $\tilde u=u$ a.e. $\Omega$ and  $\tilde u=0$ $\Phi-$q.e on $D \setminus \Omega$.
\end{itemize}

Therefore, the space $W^{1,\Phi}_0(\Omega)$ is endowed with the norm 
$$
\| u\|_{W^{1,\Phi}_0(\Omega)}=\| \tilde u\|_{W^{1,\Phi}(\R^n)}.
$$
\end{prop}

We define now the notion of relative capacity in this settings. 

Let $\Omega\subset \R^n$, $K\subset \Omega$ be compact.  Denote
$$
R_\Phi(K,\Omega):= \{ u \in W^{1,\Phi}(\Omega)\cap C(\Omega)\colon u>1 \text{ in } K \text{ and } u\geq 0\}.
$$
We define
$$
\text{cap}^*_\Phi (K,\Omega):=\inf_{ u \in R_\Phi(K,\Omega)} \rho_{\Phi,\Omega} (|\nabla u|).
$$
Further, for $U\subset \Omega$ open, we set
$$
\cp(U,\Omega):=\sup_{K\subset U, \, K \text{ compact}} \text{cap}^*_\Phi(K,\Omega),
$$
and for an arbitrary set $E\subset \Omega$ we define the \emph{relative relative $\Phi-$capacity} of $E$ with respect to $\Omega$ as
$$
\cp(E,\Omega):=\inf_{E \subset U,\, U \emph{ open }} \cp(U,\Omega).
$$ 

Among other useful properties, $\cp$  fulfills the following properties.
\begin{lema} \cite[Theorems 3.1 and 3.2]{Harjuleto-Hasto-Koskenoja}\label{prop.capa}
The function $\cp$ satisfies the following;
\begin{itemize}
\item[(i)] $\cp(\emptyset) =0$,
\item[(ii)] $\cp(E_1)\leq \cp(E_2)$ if $E_1\subset E_2\subset\R^n$,
\item[(iii)] if $\{E_k\}_{k\in\N}\subset \R^n$, then
$$
\cp\left(\bigcup_{i=1}^\infty E_k\right) \leq \sum_{i=1}^\infty \cp(E_k).
$$

\end{itemize}
\end{lema}

\section{Continuity of solutions}\label{continuidad}
Thorough this section $D\subset \R^n$ will denote a fixed open and bounded box, and $\Phi$ a Young function satisfying condition \eqref{cond}, unless otherwise requested. Given a sequence of open sets $\{\Omega_k\}_{k\in\N}\subset D$ and $f\in L^{\Phi^*}(\Omega)$, in view of Theorem \ref{teo.existencia}, we denote $u_{\Omega_k,f}$ a weak solution of 
\begin{align} \label{prob.f}
\begin{cases}
-\Delta_\phi u_{\Omega_k,f}=f &\text{ in } \Omega_k\\
u_{\Omega_k,f}=0 &\text{ on } \partial\Omega_k.
\end{cases}
\end{align}
When $f\equiv 1$, the corresponding solution $u_{\Omega_k,1}$ is usually called the \emph{torsion function}.

In this context, the sequence $\{\Omega_k\}_{k\in\N}\subset D$  is said to \emph{$\gamma-$converge} to an open set $\Omega\subset D$ (and it is denoted $\Omega_k\tog \Omega$) if for all $f\in L^{\Phi^*}(\Omega)$ the sequence of solutions $u_{\Omega_k,f}$ weakly converges in $W^{1,\Phi}_0(D)$ to the solution $u_{\Omega,f}$ of the limit problem
\begin{align} \label{prob.f.lim}
\begin{cases}
-\Delta_\phi u_{\Omega,f}=f &\text{ in } \Omega\\
u_{\Omega}=0 &\text{ on } \partial\Omega.
\end{cases}
\end{align}

Our first aim will be to prove that  in the definition of $\gamma-$convergence we can restrict ourselves to consider only  $f\equiv 1$. For that end, we need to prove several properties on solutions of \eqref{prob.f}.

We start with the following key lemma, which can be seen as a refinement of the monotonicity of Young functions.

\begin{lema} \label{lema4.1}
Assume that $\Phi$ is a Young function satisfying \eqref{cond}.  Then for all $a,b\in \R^n$ it holds that
\begin{equation}\label{des1}
\left(\frac{\phi(|a|)}{|a|}a-\frac{\phi(|b|)}{|b|}b \right)\cdot (a-b) \geq 0.
\end{equation}
Assume further \eqref{condC}, then there exists a positive constant $C=C(p^-)$ such that
\begin{equation}\label{des2}
  \left(\frac{\phi(|a|)}{|a|}a-\frac{\phi(|b|)}{|b|}b \right)\cdot (a-b)  \geq C   \Phi(|a-b|).
\end{equation}
\end{lema}
\begin{proof}
Let $a,b\in \R^n$. From the following inequality 
$$
\left(\frac{\phi(|a|)}{|a|}a-\frac{\phi(|b|)}{|b|}b \right)\cdot (a-b)\geq (|a|-|b|)(\phi(|a|)-\phi(|b|)),
$$
relation \eqref{des1} follows since $\phi$ is an increasing function.

Now we prove \eqref{des2}. A direct computation gives that
\begin{align*}
I&:= \left(\frac{\phi(|a|)}{|a|}a-\frac{\phi(|b|)}{|b|}b \right)\cdot (a-b) 
\\&=   \int_0^1 (a-b) \frac{d}{dt} \left(\frac{\phi(|at + (1-t)b|)}{|at + (1-t)b|}(at + (1-t)b) \right)\,dt \\
&=  |a-b|^2  \int_0^1  \phi'(|at + (1-t)b|)\,dt.
\end{align*}
When $|b|\geq |a-b|$ we have that
$$
|at+(1-t)b| \geq |b| - |a-b|t \geq (1-t)|a-b|
$$
from where, by using the convexity of $\phi$ we get
\begin{align*}
I&\geq |a-b|^2 \int_0^1  \phi'((1-t)|a-b|)\,dt.
\end{align*}
Observe that,  by using  Remark \ref{condic}, from the last inequality we obtain that
\begin{align*}
I&\geq p^- (p^- +1)\int_0^1 \frac{1}{(1-t)^2} \Phi((1-t)|a-b|)\,dt\\
&\geq p^- (p^- +1)\Phi(|a-b|) \int_0^1  (1-t)^{p^- -1}\,dt\\
&=(p^-+1) \Phi(|a-b|)
\end{align*}
where we have used propery \eqref{G1}.

Now, assume that $|b|< |a-b|$. In this case we have that
$$
|ta +(1-t)b| \leq  (3-t)|a-b|
$$
and then, using Remark \ref{condic}
\begin{align*}
I&\geq \int_0^1 \frac{|at+(1-t)b|^2}{|3-t|^2} \phi'(|at+(1-t)b|)\,dt\\
&\geq p^-(p^- +1 )\int_0^1 \frac{\Phi(|at+(1-t)b|)}{|3-t|^2}\,dt.
\end{align*}
In light of  Lemma \ref{Phi.tilde} the function $\tilde \Phi:=\Phi(\sqrt{t})$ is convex for $t\geq0$ and then,   from Jensen's inequality we get
\begin{align*}
I&\geq \frac{p^-(p^-+1)}{9} \int_0^1 \tilde\Phi(|at+(1-t)b|^2)\,dt\\
&\geq \frac{p^-(p^-+1)}{9}  \tilde \Phi \left(\int_0^1 |at+(1-t)b|^2\,dt \right)\\
&=\frac{p^-(p^-+1)}{9} \tilde \Phi \left(\int_0^1 t^2 |a|^2 + (1-t)^2|b|^2 + 2 t(1-t)a\cdot b\,dt \right)\\
&=\frac{p^-(p^-+1)}{9}  \tilde \Phi \left(  \frac12 |a|^2 + \frac13|b|^2 + \frac13 a\cdot b \right)\\
&\geq\frac{p^-(p^-+1)}{9}\frac{1}{3^\frac{p^- +1}{2}}  \tilde \Phi \left(|    |a|^2 +  |b|^2 +   a\cdot b | \right).
\end{align*}
Finally, since it holds that
$$
\frac{ ||a|^2 + |b|^2 + a\cdot b|}{|a-b|^2} \geq \frac14 \quad \text{ for all } a,b \text{ such that } |a-b|>|b|,
$$
from the last inequality we arrive at
$$
I\geq   \frac{p^-(p^-+1)}{9}\frac{1}{12^\frac{p^- +1}{2}}  \tilde\Phi (|a-b|^2) = C(p^-) \Phi(|a-b|)
$$
and the proof concludes.
\end{proof}

The following result provides for the monotonicity of solutions with respect to the domain.

\begin{lema} \label{monotonia} 
Let $f\in L^{\Phi^*}(D)$ be nonnegative. If $\Omega_1\subset \Omega_2$, then $u_{\Omega_1,f} \leq u_{\Omega_2,f}$.
\end{lema}

\begin{proof}
We denote $u_k= u_{\Omega_k,f}$. Since $u_k \in W^{1,\Phi}_0(\Omega_k)$ is a solution, then
$$
\int_{\Omega_k} \frac{\phi(|\nabla u_k|)}{|\nabla u_k|} \nabla u_k\cdot \nabla v \,dx = \int_{\Omega_k} f v\,dx \qquad \forall v \in W^{1,\Phi}_0(\Omega_k).
$$
The inclusion $W^{1,\Phi}_0(\Omega_1)\subset W^{1,\Phi}_0(\Omega_2)$ gives that
$$
\int_{\Omega_1} \left(\frac{\phi(|\nabla u_1|)}{|\nabla u_1|} \nabla u_1 -\frac{\phi(|\nabla u_2|)}{|\nabla u_2|} \nabla u_2 \right) \cdot \nabla v \,dx =0  \qquad \forall v \in W^{1,\Phi}_0(\Omega_1).
$$
Since $f\geq 0$, from Proposition  \ref{princ.max.fuerte} we obtain that $u_2\geq 0$ in $\Omega_2$. Hence $(u_1-u_2)^+ \leq u_1^+ \in W^{1,\Phi}_0(\Omega_1)$ and then for all $v \in W^{1,\Phi}_0(\Omega_1)$ it holds that
$$
\int_{\Omega_1} \left(\frac{\phi(|\nabla u_1|)}{|\nabla u_1|} \nabla u_1 -\frac{\phi(|\nabla u_2|)}{|\nabla u_2|} \nabla u_2 \right)\cdot \nabla (u_1-u_2)^+ \,dx =0 .
$$
Moreover, for all $v \in W^{1,\Phi}_0(\Omega_1)$ we have that
$$
\int_{\Omega_1 \cap \{u_1>u_2\}} \left(\frac{\phi(|\nabla u_1|)}{|\nabla u_1|} \nabla u_1 -\frac{\phi(|\nabla u_2|)}{|\nabla u_2|} \nabla u_2 \right)\cdot \nabla (u_1-u_2)^+ \,dx =0 .
$$
But, since in $\Omega_1\cap \{u_1>u_2\}$ we have that $\nabla (u_1-u_2)=\nabla (u_1-u_2)^+$, from where
$$
\int_{\Omega_1 \cap \{u_1>u_2\}}  \Phi(|\nabla (u_1-u_2)|)\,dx \leq 0
$$
but, since $\nabla (u_1-u_2)^+ = \nabla (u_1-u_2)\chi_{u_1>u_2}$ we can rewrite the last equation as
$$
\int_{\Omega_1}  \Phi(|\nabla (u_1-u_2)^+|) \,dx= \int_{\Omega_1}  \Phi(|\nabla (u_1-u_2)|) \chi_{u_1>u_2}\,dx \leq 0.
$$
Hence, $\nabla (u_1-u_2)^+ = 0$ in $\Omega_1$, and then $u_1-u_2$ is constant in $\Omega_1$. Since $(u_1-u_2)^+ \in W^{1,\Phi}_0(\Omega_1)$, $(u_1-u_2)^+=0$. Therefore, $u_1 -u_2 \leq 0$ and the proof concludes.
\end{proof}

The following result establishes the uniform boundedness of a sequence of solutions in the Orlicz-Sobolev norm.

\begin{lema} \label{lema.acotada}
Given $f\in L^{\Phi^*}(D)$ and let $\{u_{\Omega_k,f}\}_{k\in\N}\subset W^{1,\Phi}_0(\Omega_k)$ be a sequence of solutions of \eqref{prob.f}, then 
$$
\|\nabla u_{\Omega_k,f}\|_{L^\Phi(\Omega_k)}\leq C,
$$
where $C$ is a uniform positive constant depending only on $f$, $D$ and $p^\pm$.
\end{lema}
\begin{proof}
Denote $u_k=u_{\Omega_k,f}$. If $\|\nabla u_k\|_{L^\Phi(\Omega_k)} \leq 1$ there is nothing to do. Assume otherwise that $\|\nabla u_k\|_{L^\Phi(\Omega_k)} > 1+\ve$ for some $\ve>0$, then
\begin{align} \label{eqq1}
\begin{split}
\int_{\Omega_k} \Phi(|\nabla u_k|)\,dx &\geq \left( \frac{\|\nabla u_k\|_{L^\Phi(\Omega_k)}}{1+\ve}\right)^{p^- +1} \int_{\Omega_k} \Phi\left( |\nabla u_k| \frac{1+\ve}{\|\nabla u_k \|_{L^\Phi(\Omega_k)}} \right) \,dx \\
 &\geq \left( \frac{\|\nabla u_k\|_{L^\Phi(\Omega_k)}}{1+\ve}\right)^{p^-+1} \int_{\Omega_k} \Phi\left(  \frac{|\nabla u_k|}{\|\nabla u_k \|_{L^\Phi(\Omega_k)}} \right) \,dx\\ 
 &=\left( \frac{\|\nabla u_k\|_{L^\Phi(\Omega_k)}}{1+\ve}\right)^{p^-+1}
 \end{split}
\end{align}
where we have used the definition of the Luxemburg norm.

Observe that since $u_k$ is solution then, by using \eqref{cond} we get
$$
(p^- +1) \int_{\Omega_k} \Phi(|\nabla u_k|)\,dx \leq \int_{\Omega_k} \phi(|\nabla u_k|)|\nabla u_k| \,dx = \int_{\Omega_k} fu_k\,dx
$$
but, from H\"older's and Poincar\'e's inequalities for Orlicz functions we get that
\begin{align*}
\int_{\Omega_k} f u_k\,dx &\leq \|f\|_{L^{\Phi^*}(\Omega_k)} \| u_k\|_{L^{\Phi}(\Omega_k)}	\\
&\leq C\|f\|_{L^{\Phi^*}(\Omega_k)} \| \nabla u_k\|_{L^{\Phi}(\Omega_k)}
\end{align*}
where, from Remark \ref{ctePoincare} the constant $C$ only dependes of the parameters $p^\pm$ and the diameter of $D$. From \eqref{eqq1} and the last two relations we get
\begin{align*}
\|\nabla u_k\|_{L^\Phi(\Omega_k)} &\leq  (1+\ve) \left( \int_{\Omega_k} \Phi(|\nabla u_k|)\,dx  \right)^\frac{1}{p^- +1}\\
&\leq    (1+\ve) \left( \frac{C}{p^-+1} \|f\|_{L^{\Phi^*}(\Omega_k)} \|\nabla u_k\|_{L^\Phi(\Omega_k)} \right)^\frac{1}{p^-+1}
\end{align*}
and finally, 
$$
\|\nabla u_k\|_{L^\Phi(\Omega_k)} \leq  \tilde{C} \|f\|_{L^{\Phi^*}(D)}^{\frac{1}{p^-}},
$$
with $\tilde{C}=\tilde{C}(p^\pm,D)$, concluding the proof.
\end{proof}

We prove now a stability result for solutions of \eqref{prob.f} with respect to $f$. 

\begin{prop} \label{estabilidad}
Let $\Omega\subset D\subset \R^n$ be open and $f_1,f_2\in L^{\Phi^*}(D)$. Then there exists $C=C(D,p^-,f_1,f_2)$ independent of $\Omega$ such that
$$
\int_\Omega\Phi(|\nabla u_{\Omega,f_1} - \nabla u_{\Omega,f_2}|)\leq C \|f_1-f_2\|_{L^{\Phi^*}(D)}
$$
and 
$$
\| \nabla u_{\Omega,f_1} - \nabla u_{\Omega,f_2} \|_{L^\Phi(\Omega)} \leq C \|f_1-f_2\|_{L^{\Phi^*}(D)}.
$$
\end{prop}
\begin{proof}
Let us denote $u_k = u_{\Omega,f_k}$. Since for every $\psi\in W^{1,\Phi}_0(\Omega)$, $u_k$, $i=1,2$ satisfy that
$$
\int_\Omega  \frac{\phi(|\nabla u_k|)}{|\nabla u_k|} \nabla u_k \cdot \nabla \psi\, dx=\int_\Omega f_k \psi\, dx,
$$
testing with $\psi=u_1-u_2$ and subtracting we get
\begin{align*}
\int_\Omega &\left (\frac{\phi(|\nabla u_1|)}{|\nabla u_1|} \nabla u_1-\frac{ \phi(|\nabla u_2|)}{|\nabla u_2|} \nabla u_2 \right) \cdot (\nabla u_1-\nabla u_2)dx\\
 &=\int_\Omega (f_1-f_2)(u_1-u_2)dx\\
&\leq \|f_1-f_2\|_{L^{\Phi^*}(D)} \| u_1-u_2\|_{L^{\Phi}(\Omega)}\\
&\leq  C_P(D)\|f_1-f_2\|_{L^{\Phi^*}(D)} \| \nabla u_1- \nabla u_2\|_{L^{\Phi}(\Omega)}\\
&\leq  C_P(D)\|f_1-f_2\|_{L^{\Phi^*}(D)} ( \|  \nabla u_1\|_{L^{\Phi}(\Omega)} + \| \nabla u_2\|_{L^{\Phi}(\Omega)})\\
&\leq C  \|f_1-f_2\|_{L^{\Phi^*}(D)}
\end{align*}
where we have used Lemma \ref{poincare}, Remark \ref{ctePoincare} and  Proposition \ref{lema.acotada}. Finally by Lemma \ref{lema4.1} we get the result.
\end{proof}

Now are now in position to prove that it suffices with considering $f=1$ when studying $\gamma-$convergence of domains.

\begin{prop} \label{prop.equiv}
Let $\{\Omega_k\}_{k\in\N}$ and $\Omega$ be open subsets of $D\subset \R^n$ and let $\Phi$ be a Young function satisfying \eqref{cond} and \eqref{condC}. Then, for every $f\in L^{\Phi^*}(D)$ we have  $u_{\Omega_k,f} \cd u_{\Omega,f}$ weakly  in $W^{1,\Phi}_0(D)$ if and only if $u_{\Omega_k,1}\cd u_{\Omega,1}$ weakly in $W^{1,\Phi}_0(D)$.
\end{prop}
\begin{proof}
It is only necessary to prove one implication. Indeed, assume that $u_{\Omega_k,1} \cd u_{\Omega,1}$ weakly in $W^{1,\Phi}_0(D)$. Let us show that given $f\in L^{\Phi^*}(D)$ it holds that $u_{\Omega_k,f} \cd  u_{\Omega,f}$ weakly in $W^{1,\Phi}_0(D)$.

First, observe that in the light of Lemma \ref{lema.acotada} the sequence $u_{\Omega_k,f}$ is uniformly bounded in $W^{1,\Phi}_0(D)$, then, due to the reflexivity of $W^{1,\Phi}_0(D)$, up to a subsequence, there exists a function $v\in W^{1,\Phi}_0(\Omega)$ such that
\begin{equation*} 
u_{\Omega_k,f} \cd v \text{ weakly in }W^{1,\Phi}_0(D),
\end{equation*}
but moreover, since $\{u_k\}_{k\in\N}$ are solutions, it easily follows from  Lemma \ref{lema4.1} that $u_{\Omega_k,f}\to v$ a.e. in $D$, up to a subsequence if necessary.

 From mentioned convergences, it only remains to be proved that $v=u_{\Omega,f}$, and by uniqueness of the limit, we will have convergence of the whole sequence strongly in $L^\Phi(D)$.

Observe that by Lemma \ref{densidad}, it suffices with considering $f\in L^\infty(D)$. Indeed, since $|f|\leq M$, by Lemma \ref{comparison} we get
\begin{align} \label{estim1}
-M u_{\Omega_k,1}\leq u_{\Omega_k,f} \leq M u_{\Omega_k,1} \quad a.e.\text{ in }D.
\end{align}
Taking limit as $k\to\infty$, by using our hypothesis we obtain that
$$
-M u_{\Omega,1} \leq v \leq M  u_{\Omega,1}\quad a.e.\text{ in }D.
$$
By the above inequality we have that $v\in W^{1,\Phi}_0(D)$.

\medskip

Let us see that $v = u_{\Omega,f}$. Given $f\in L^{\Phi^*}(D)$, by Lemma \ref{densidad} there exists a sequence $\{f_j\}_j\subset L^\infty(D)$ such that $f_j\to f$ in $L^{\Phi^*}(D)$. 
Then, given $\psi\in L^{\Phi^*}(D)$,
\begin{align*}
\int_D (u_{\Omega_k,f}-u_{\Omega,f})\psi\,dx  &=  \int_D (u_{\Omega_k,f}-u_{\Omega_k,f_j})\psi\,dx + \int_D (u_{\Omega_k,f_j}-u_{\Omega,f_j})\psi\,dx\\ 
&+ \int_D (u_{\Omega,f_j}- u_{\Omega,f})\psi\,dx.
\end{align*}
Observe that by using H\"older inequality for Young functions,  Lemma \ref{poincare} and Proposition  \ref{estabilidad}, we have that given $\ve>0$ there exists $j_0\in\N$ such that
\begin{align*}
\int_D (u_{\Omega_k,f}-u_{\Omega_k,f_j})\psi\,dx &\leq C \|   \nabla u_{\Omega_k,f}-\nabla u_{\Omega_k,f_j} \|_{L^{\Phi^*}(D)} \|\psi\|_{\Phi^*}\\&\leq C \|f-f_j\|_{L^{\Phi^*}(D)} \|\psi\|_{L^{\Phi^*}(D)} <\ve
\end{align*}
and
\begin{align*}
\int_D (u_{\Omega ,f}-u_{\Omega	,f_j})\psi\,dx &\leq C \| \nabla_{\Omega,f}-\nabla u_{\Omega,f_j} \|_{L^{\Phi^*}(D)} \|\psi\|_{L^{\Phi^*}(D)}\\ &\leq C \|f-f_j\|_{L^{\Phi^*}(D)} \|\psi\|_{L^{\Phi^*}(D)} <\ve
\end{align*}
uniformly in $k\in\N$ for every $j\geq j_0$, where $C=C(p^\pm,D)$. Moreover, by hypothesis and \eqref{estim1} we get
$$
 \int_D (u_{\Omega_k,f_j}-u_{\Omega,f_j})\psi\,dx \to 0 \quad \text{ as } k\to\infty.
$$
Gathering the last four relations we get that $v=u_{\Omega,f}$ and the proof concludes.
\end{proof}

In order to study the $\gamma-$convergence of the sequence of domains $\{\Omega_k\}_{k\in\N}$ we require some geometrical condition on the sequence of domains. 

With that end, we recall the following notion of set convergence.

\begin{defi}\label{Hausdorff}
We recall that the \emph{Hausdorff complementary topology} on $\mathcal{A}:=\{\Omega\colon \Omega\subset D, \Omega \text{ open}\}$ is given by the metric
$$
d_{H^c}(\Omega_1,\Omega_2)=d(\Omega_1^c,\Omega_2^c)
$$
where $d$ is the usual Hausdorff distance.

Finally, we say that $\{\Omega_k\}_{k\in \N}$ {\em converges to $\Omega$ in the sense of the Hausdorff complementary topology}, denoted by $\Omega_k\toh \Omega$ if $d_{H^c}(\Omega_k,\Omega)\to 0$ as $k\to \infty$.
\end{defi}
The following characterization of the Hausdorff complementary convergence is well-known. See \cite{Bucur-Butazzo} for a proof.
\begin{lema} \label{h.caract}
 If $\Omega_k \toh \Omega$, then for every compact set $K\subset \Omega$, there exist $N_K\in\N$ such that for every $k\geq N_K$ we have $K\subset \Omega_k$.
\end{lema}

The following observation will be useful in our arguments.
\begin{lema} \label{grad.acotado}
The sequence $\left\{  \phi(|\nabla u_{\Omega_k,f}|) \right\}_{k\in\N}$ is bounded in $L^{\Phi^*}(D)$.
\end{lema}
\begin{proof}
Observe that Lemma \ref{lemita} gives that
\begin{align*}
\int_\Omega \Phi^*\left( |\phi(|\nabla u_{\Omega_k,f}|)| \right)\,dx \leq (p^+ +1) \int_\Omega \Phi(|\nabla u_{\Omega_k,f}|)\,dx.
\end{align*}
Therefore the result follows from Lemma \ref{lema.acotada}.
\end{proof}
The following result about the limit of solutions  is a first step in proving the continuity of solutions.

\begin{lema} \label{lema.conv.fuerte}
Let $\{u_{\Omega_k,f}\}_{k\in\N}$ be a sequence of solutions. Then, up to a subsequence, $u_{\Omega_k,f}\cd v$ weakly in $W^{1,\Phi}_0(D)$. Morevoer, $u_{\Omega_k,f}\to v$ strongly in $L^\Phi(D)$.
\end{lema}
\begin{proof}
Denote $u_k:=u_{\Omega_k,f}$. Observe that by Lemma \ref{lema.acotada}, $u_k$ is uniformly bounded in $W^{1,\Phi}_0(D)$, then,  up to a subsequence if necessary, we have that 
\begin{align} \label{cdebil}
u_k\cd v \quad \text{ weakly in } W^{1,\Phi}_0(D)
\end{align}
to some function $v\in W^{1,\Phi}_0(D)$.

Now, if 
$I:=\int_1^\infty  s^{-(1+1/n)}\Phi^{-1}(s)\,ds = \infty$, by Theorem \ref{compacidad} item (a) we have that the convergence of $u_k$ to $v$ is strong in $L^\Phi(D)$. 

If otherwise $I<\infty$ we proceed as follows. By Theorem \ref{compacidad} item (b), $W^{1,\Phi}(D)$ is compactly  embedded in the H\"older space $C^{0,\sigma(t)}(\bar D)$ with modulus of continuity $\sigma(t)$  given by
$$
\sigma(t)=\int_{t^{-n}}^\infty \frac{\Phi^{-1}(s)}{s^{1+1/n}}\,ds.
$$
Since for $t\geq 1$ condition \eqref{cond} gives that
$$
\frac{\Phi^{-1}(t)}{t^{1+1/n}} \geq t^{\frac{1}{p^+ +1}-1-\frac{1}{n}},
$$
it  follows that $C^{0,\sigma(t)}(\bar D)\subset C^{0,\tilde\sigma}(\bar D)$, where $\tilde\sigma:=1-\frac{n}{p^++1}$.   But, since $C^{0,\tilde \sigma }(\bar D)\subset L^{p^-}(D)  \subset L^\Phi(D)$ (see \cite[Theorem 3.17.7]{Fucik-John-Kufner}) we get that $W^{1,\Phi}(D)\subset\subset L^\Phi(D)$ and therefore $u_k\to v$ strongly in $L^\Phi(D)$.
\end{proof}

\begin{prop} \label{prop.previa}
Assume that $f\in L^{\Phi^*}(D)$ and let $\Phi$ be a Young function satisfying \eqref{cond} and \eqref{condC}. If $\Omega_k\toh\Omega$ then, up to a subsequence, $u_{\Omega_{k,f}}\cd v$, where $v\in W^{1,\Phi}_0(D)$.  satisfies the equation $-\Delta_\phi v=f$ in weak sense.	
\end{prop}

\begin{proof}
Denote $u_k=u_{\Omega_k,f}$. Observe that by  Lemma \ref{lema.acotada} and  Lemma \ref{lema.conv.fuerte},  up to a subsequence if necessary, we have that 
\begin{align} \label{cdebil}
u_k\cd v \quad \text{ weakly in } W^{1,\Phi}_0(D)
\end{align}
and
\begin{equation} \label{c.fuerte}
u_k\to v \quad \text{ strongly in } L^\Phi(D)
\end{equation}
to some function $v\in W^{1,\Phi}_0(D)$.

In order to conclude the proposition we have to prove that for every $\psi\in C_c^\infty(\Omega)$ it holds that
$$
\int_{\Omega} \frac{\phi(|\nabla v|)}{|\nabla v|}\nabla v \cdot \nabla \psi\, dx=\int_\Omega f \psi\, dx.
$$
Observe that it suffices to prove that $\frac{\phi(|\nabla u_k|)}{|\nabla u_k|}\nabla u_k \cd  \frac{\phi(|\nabla v|)}{|\nabla v|}\nabla v$ weakly in $L^{\Phi^*}(\Omega)$.

Let $\psi\in C_c^\infty(\Omega)$. Since $K:=\supp \psi$ is compact and $\Omega_k\toh \Omega$, by Lemma \ref{h.caract} exists a $N$ such that for all $k\geq N_0$, $K \subset \Omega_k$.

Set $K^\ve=\{x\in\R^n\colon d(x,K)<\ve\}$ with $\ve$ small enough such that $K_\ve\subset \subset \Omega_k\cap \Omega$ for every $k\geq N_1$.

From now on, we consider $k\geq \max\{N_0,N_1\}$. Let $\eta\in C_c^\infty(\Omega)$ such that $\eta=1$ in $K^\frac{\ve}{2}$, $\eta=0$ in $(K^\ve)^c$ and $0\leq \eta\leq 1$ and consider the test function $\psi_k=\eta(u_k-v) \in W^{1,\Phi}_0(K^\ve)$. Then
$$
\int_{K^\ve} \frac{\phi(|\nabla u_k|)}{|\nabla u_k|}\nabla u_k \cdot \nabla \psi_k\, dx=\int_{K^\ve} f \psi_k\, dx,
$$
from where
\begin{align*}
\int_{K^\ve} \frac{\phi(|\nabla u_k|)}{|\nabla u_k|}\nabla u_k \cdot \eta &\nabla (u_k-v)\, dx =\int_{K^\ve} f \psi_k\, dx\\ &- \int_{K^\ve} \frac{\phi(|\nabla u_k|)}{|\nabla u_k|}\nabla u_k \cdot (u_k-v) \nabla \eta \, dx.
\end{align*}
Since $u_k\cd v$ in  $W^{1,\Phi}_0(D)$ we have that $\int_{K^\ve} f \psi_k\,dx\to 0$. The second integral in the right hand side of the equality above can be bounded by using H\"older's inequality as follows
$$
\|\nabla \eta\|_{L^\infty(D)}   \| \phi(|\nabla u_k|)\|_{L^{\Phi^*}(D)} \| u_k-v\|_{L^\Phi(D)}\to 0,
$$
where the convergence holds by using \eqref{c.fuerte} and fact that   $\| \phi(|\nabla u_k|)\|_{\Phi^*}$ is bounded in light of Lemma \ref{grad.acotado}. Therefore, 
$$
\limsup_{k\to\infty}\int_{K^\ve} \frac{\phi(|\nabla u_k|)}{|\nabla u_k|}\nabla u_k \cdot \eta \nabla (u_k-v)\, dx \leq 0.
$$
On the other hand, since \eqref{cdebil} holds, 
$$
\int_{K^\ve} \frac{\phi(|\nabla v|)}{|\nabla v|}\nabla v \cdot \eta \nabla (u_k-v)\, dx\to 0.
$$
therefore, subtracting the last two relation we get
$$
\limsup_{k\to\infty}\int_{K^\ve} \left( \frac{\phi(|\nabla u_k|)}{|\nabla u_k|}\nabla u_k - \frac{\phi(|\nabla v|)}{|\nabla v|}\nabla v  \right) \cdot \eta \nabla (u_k-v)\, dx \leq 0.
$$
Since $K^\frac{\ve}{2}\subset K^\ve$ and $\eta=1$ in $K^\frac{\ve}{2}$, from Lemma \ref{lema4.1} we have
$$
\lim_{k\to\infty}\int_{K^\frac{\ve}{2}} \left( \frac{\phi(|\nabla u_k|)}{|\nabla u_k|}\nabla u_k - \frac{\phi(|\nabla v|)}{|\nabla v|}\nabla v  \right) \cdot   \nabla (u_k-v)\, dx =0.
$$

Again the light of Lemma \ref{lema4.1}, \cite[Lemma 3.10.4]{Fucik-John-Kufner} and \cite[3.18.5]{Fucik-John-Kufner} follows
%and again, from Lemma \ref{lema4.1} it follows that $ \phi(|\nabla u_k|)\frac{\nabla u_k}{|\nabla u_k|} - \phi(|\nabla v|)\frac{\nabla v}{|\nabla v|}  \to 0$ in $L^1(K^\frac{\ve}{2})$ and therefore a.e. in $K^\frac{\ve}{2}$. 
%
%From these facts it is easy to see that
\begin{equation} \label{conv.ae}
\nabla u_k\to \nabla v \quad\text{  a.e. in } K^\frac{\ve}{2}.
\end{equation}

Finally, since $\{  \phi(|\nabla u_k|) \frac{ \nabla u_k}{|\nabla u_k|} \}_{k\in\N}$ is bounded in $L^{\Phi^*}(D)$ by Proposition \ref{grad.acotado}, there exists $\xi\in (L^\Phi(K^\frac{\ve}{2}))^n$ such that
$$
\frac{\phi(|\nabla u_k|)}{|\nabla u_k|} \nabla u_k \cd \xi \quad \text{ weakly in } L^{\Phi^*}(K^\frac{\ve}{2}).
$$

The almost everywhere convergence stated in \eqref{conv.ae} allows us to conclude that $\xi = \phi(|\nabla v)\frac{ \nabla v}{|\nabla v|}$ in $K^\frac{\ve}{2}	$ and that
$$
\int_{K^{\frac{\ve}{2}}} \frac{\phi(|\nabla u_k|)}{|\nabla u_k|}\nabla u_k \cdot \nabla \psi\, dx\to  \int_{K^{\frac{\ve}{2}}} \frac{\phi(|\nabla v|)}{|\nabla v|}\nabla v \cdot \nabla \psi\, dx. 
$$
Since $\supp(\nabla \psi)\subset K\subset K^\frac{\ve}{2}\subset K^\ve\subset \Omega_k\cap \Omega$, we get that
$$
\int_{\Omega_k} \frac{\phi(|\nabla u_k|)}{|\nabla u_k|}\nabla u_k \cdot \nabla \psi\, dx\to  \int_{\Omega} \frac{\phi(|\nabla v|)}{|\nabla v|}\nabla v \cdot \nabla \psi\, dx
$$
concluding the proof.
\end{proof}

\color{black}
It only remains to be seen that the limit function $u_\Omega$ satisfies the boundary condition. For a general sequence of domains this could be false. Therefore, that conclusion will follow by assuming a further  capacitary condition on the sequence of domains and by using the  characterization of Orlicz-Sobolev spaces stated in Proposition \ref{eqiv.espacios} given in terms of the $\Phi-$capacity.

We are finally in position to prove our main result.
\begin{teo}\label{main}
 Let $\Phi$ be a Young function satisfying \eqref{cond} and  \eqref{condC}. Assume that $f\in L^{\Phi^*}(D)$ and let $\Omega_k,\Omega\subset D$. If $\Omega_k\toh\Omega$ and $\cp(\Omega_k\setminus \Omega,D)\to 0$, then, up to a subsequence, $u_{\Omega_{k,f}}\cd u_{\Omega,f}$ weakly in $W^{1,\Phi}_0(D)$. 
\end{teo}

\begin{proof}
Observe that from Proposition \ref{prop.previa} there exists a function $v\in W^{1,\Phi}_0(D)$ such that $u_k \cd v$ in $W^{1,\Phi}_0(D)$ and $v$ solves $-\Delta_\phi v = f$ in $\Omega$ in the weak sense. Therefore it only remains to be proved the boundary condition $v=0$ on $\partial \Omega$. But by Proposition \ref{eqiv.espacios} it suffices with proving that $\tilde v = 0$ $\Phi-$q.e. in $\Omega^c$.

Consider $\tilde \Omega_j=\cup_{k\geq j} \Omega_k$ and $E=\cap_{j\geq 1} \tilde \Omega_j$.

Since $u_k\cd v$ in $W^{1,\Phi}_0(D)$, by using Mazur's Lemma (see for instance \cite{Ekeland}) there exists a sequence $v_j=\sum_{k=j}^{\ell_j} a_{k_j} u_k$ such that $a_{k_j}\geq 0$, $\sum_{k=j}^{\ell_j} a_{k_j}=1$ and $v_j \to v$ strongly in $W^{1,\Phi}_0(D)$.

By Proposition \ref{eqiv.espacios}, the functions $u_k\in W^{1,\Phi}_0(\Omega_k)$ have a $\Phi-$q.e. representative $\tilde u_k$ such that $\tilde u_k=0$ $\Phi-$q.e. $\Omega_k^c$. Hence, $\tilde v_j= \sum_{k=j}^{\ell_j} a_{k_j} \tilde u_k =0$ $\Phi-$q.e. $\cap_{k=j}^{\ell_j} \Omega_k^c \supset \tilde \Omega_j^c$ for every $j\geq 1$. Then $\tilde v_j=0$ $\Phi-$q.e. in $\tilde \Omega_j^c$ or every $j\geq 1$.  From this, $\tilde v_j =0$ $\Phi-$q.e. $\cup_{j\geq 1}\tilde \Omega_j^c = E^c$.

Since $v_j\to v$ in $W^{1,\Phi}_0(D)$, by Lemma  \ref{prop.c}, $\tilde v_k \to \tilde v$ $\Phi-$q.e. Therefore, $\tilde v=0$ $\Phi-$q.e. in $E^c$. By our hypothesis we can assume, up to a subsequence if necessary,  that  $\cp(\Omega_k\setminus \Omega,D)\leq \frac{1}{2^k}$. Therefore, by using Lemma \ref{prop.capa}, 
\begin{align*}
\cp(\tilde \Omega_j\setminus \Omega,D) &= \cp(\cup_{k\geq j} \Omega_k \setminus \Omega,D)\\
&\leq \sum_{k\geq j} \cp(\Omega_k\setminus \Omega,D) \leq \sum_{k\geq j} \frac{1}{2^k} = \frac{1}{2^{j-1}}.
\end{align*}

Finally, since $E\subset \tilde \Omega_j$ we obtain that $E\setminus \Omega\subset \tilde \Omega_k\setminus \Omega$ for every $j\geq 1$ and, from the last relation we get that
$$
\cp(E\setminus \Omega,D) \leq \cp(\tilde \Omega_j\setminus \Omega,D) \leq \frac{1}{2^{j-1}} \qquad \text{for every }j\geq 1.
$$
By taking limit as $j\to\infty$ we arrive at the equality 
$$
\cp(E\setminus \Omega,D)= \cp(\Omega^c \setminus E^c,D)=0,
$$
 which allow us to conclude that $\tilde v=0$ $\Phi-$q.e. in $\Omega^c$, which finishes the proof.
\end{proof}

\begin{remark} \label{rem.fuerte}
Observe that the convergence of the sequence of solutions $\{u_{\Omega_k,f}\}_{k\in\N}$ to the limit solution $u_{\Omega,f}$ obtained in Theorem \ref{main} is indeed strong in $W^{1,\Phi}_0(D)$. In effect, since $u_k:=u_{\Omega_k,f}$ and $u:=u_{\Omega,f}$ are  weak solutions of \eqref{prob.f} and \eqref{prob.f.lim}, respectively, it follows that
\begin{align*}
&\int_D \left(\frac{\phi(|\nabla u_k|)}{|\nabla u_k|}\nabla u_k  - \frac{\phi(|\nabla u|)}{|\nabla u|}\nabla u \right)\cdot (\nabla u_k-\nabla u)\,dx=\\ 
&-\int_D \left(\frac{\phi(|\nabla u_k|)}{|\nabla u_k|}\nabla u_k  - \frac{\phi(|\nabla u|)}{|\nabla u|}\nabla u \right)\cdot \nabla u\,dx - \int_D \phi(|\nabla u|)\frac{\nabla u}{|\nabla u|} \cdot (\nabla u_k-\nabla u)\,dx\\
&+\int_D f(u_k-u)\,dx.
\end{align*}
From Lemma \ref{lema4.1} and Theorem \ref{main}, it follows that
$$
\lim_{k\to \infty} \int_D \Phi(|\nabla u_k - \nabla u|)\,dx = 0, 
$$
giving the desired convergence in light of \cite[Lemma 3.10.4]{Fucik-John-Kufner}.
\end{remark}

We provide now a condition independent of the capacitary condition  which ensures the $\gamma-$convergence of the sequence of domains.

\begin{teo}\label{main.2}
 Assume that the design box $D\subset \R^n$ is Lipschitz. Let $\Phi$ be a Young function satisfying \eqref{cond} and \eqref{condC} such that
\begin{equation}\label{cond.morrey}
\int_1^\infty \frac{\Phi^{-1}(t)}{t^{1+1/n}}\, dt <\infty,
\end{equation} 
and let $f\in L^{\Phi^*}(D)$.
Assume that the sequence of domains $\{\Omega_k\}_{k\in\N}\subset D$ verifies that $\Omega_k \toh \Omega\subset D$, then $u_{\Omega_k,f}\cd u_{\Omega,f}\text{ weakly in } W^{1,\Phi}_0(D)$.
\end{teo}

\begin{proof}
By Proposition \ref{prop.equiv} we can consider $f\equiv 1$ and assume that $u_{\Omega_k,1}\cd v$ weakly in $W^{1,\Phi}_0(D)$.

Moreover, applying Proposition  \ref{prop.previa} we  conclude that $v$ fulfills the limit equation $-\Delta_\phi v = f$ in the weak sense. In order to conclude the result let us see that $v$ satisfies the boundary condition. By Proposition \ref{eqiv.espacios}  it is enough to see that $\tilde v=0$ $\Phi-$q.e. in $\Omega^c$.

From the maximum principle stated in Proposition \ref{princ.max.fuerte} together with Proposition \ref{monotonia} we have that 
\begin{equation} \label{desig99}
	0\leq u_{\Omega_k,1} \leq  u_{D,1}. 
\end{equation}

As we mentioned, by our regularity assumptions on $\Phi$ and $\Omega$, $u_{D,1}$ has a modulus of continuity $\sigma(t)$ in $\bar D$ and  $u_{D,1}\in C^{0,\sigma(t)}(\bar D)$. Therefore, 
 given $x\in \Omega^c$ and $y\in \partial D$, we have that
 $$
 u_{D,1}(x)=|u_{D,1}(x)-u_{D,1}(y)|\leq C(D) \sigma(|x-y|) \leq C(D) \, \sigma(\text{diam}(D)),
 $$
 which together with \eqref{desig99} gives that $\{u_{\Omega_k,1}\}_{k\in\N}$ is uniformly bounded. Hence, since moreover this sequence is H\"older continuous, it is uniformly equicontinuous. Then, $u_{\Omega_k,1}\to v$ uniformly, and then $\tilde u_{\Omega_k,1} \to \tilde v$.
 
 Let us see that $\tilde v=0$ $\Phi-$q.e in $\Omega^c$. Indeed, given $x\in \Omega^c$, since $\Omega_k\toh \Omega$, there exists a sequence $x_k\in \Omega_k^c$ such that $x_k\to x$. Then, by the uniform convergence, $\tilde u_{\Omega_k,1}(x_k) \to \tilde v(x)$. Since $\supp  \tilde u_k\subset \bar \Omega_k$, $\tilde u_k(x_k)=0$ for any $k\in\N$, therefore $\tilde v=0$ concluding the proof. 
\end{proof}
\begin{remark}
Observe that if $p^-+1>n$, then  condition 	\eqref{cond.morrey} is fulfilled.
\end{remark}

\begin{remark} \label{remar}
	We cannot guarantee that condition \eqref{cond.morrey} is the optimal one in Theorem \ref{main.2}. On the counterpart of the case of powers dealt  in \cite{Bucur-Trebeschi}, we conjecture that condition \eqref{cond.morrey} can be improved (up to $p^->n$) if additionally is required that the number of connected components of each  $\Omega_k^c$ is uniformly bounded.
\end{remark}

\section{Extensions and final remarks} \label{sec.rem}
\subsection{Some particular cases}
For some particular configurations of sequence of open sets  $\{\Omega_k\}_{k\in\N}$ contained in the design box $D\subset \R^n$  it can be easily checked that   we can get rid of the capacitary condition required in Theorem \ref{main}.
\begin{teo} 
 Let $\Phi$ be a Young function satisfying \eqref{cond} and \eqref{condC}. Assume that $f\in L^{\Phi^*}(D)$ and let $\Omega_k,\Omega\subset D$. If one of the following conditions are fulfilled
 
 \begin{itemize}
 \item[(a)] for every $k\in\N$, $\Omega_k$ is uniformly Lipschitz, i.e., $\partial\Omega_k$ is is locally the graph of a Lipschitz function which can be taken uniform over all the Lipschitz constants all the sets $\Omega_k$,
 
 \item[(b)] $\Omega_k,\Omega$ are convex open sets,
 
 \item[(c)] $\{\Omega_k\}_{k\in\N}\subset \Omega$,
 \end{itemize}
 and   $\Omega_k\toh\Omega$, then, up to a subsequence, $u_{\Omega_{k,f}}\cd u_{\Omega,f}$ weakly in $W^{1,\Phi}_0(\Omega)$ (and therefore strongly by Remark \ref{rem.fuerte}).
\end{teo}

\subsection{A generalization}
In view of of \cite[Theorem 1]{Montenegro} and \cite[Theorem 7.2.14]{Fucik-John-Kufner}, given a Young function $\Phi$ fulfilling \eqref{cond} and \eqref{condC},  Theorems \ref{main} and \ref{main.2} remain true for the more general operator $-\Delta_\phi (u) + d(x) \phi(|u|)$  defined in $W^{1,\Phi}_0(\Omega)$, where $d\geq 0$ is a bounded function.

Moreover, given $\Omega\subset \R^n$, and a Young function $\Phi$ satisfying \eqref{cond},   from the proof of Theorem \ref{main} it follows that the  theorem remains true for a more general class of operators   $\mathcal{F}\colon W^{1,G}_0(\Omega)\to \R$ of the form  $\mathcal{F}:=\diver g(x,\nabla u)$ whenever
\begin{itemize}
\item[(a)] $\mathcal{F}$ satisfies a strong maximum principle for supersolutions  in $\Omega$, 
\item[(b)] for any $a,b\in\R^n$ it holds that $(\phi(x,a)-\phi(x,b))\cdot (a-b)\geq C \Phi(|a-b|)$ for some universal constant $C$.
\end{itemize}

\subsection{Extension to nonlocal operators}

Recently, in \cite{FBS} it was introduced the non-local counterpart of the $\phi-$Laplacian treated in this manuscript. More precisely, given a Young function $\Phi$ satisfying \eqref{cond}, it can be considered the non-local non-standard growth operator
$$
(-\Delta_\phi)^s u:=2\, \text{p.v.} \int_{\R^n} \phi\left( |D_s u|\right)\frac{D_s u}{|D_s u|} \frac{dy}{|x-y|^{n+s}},
$$
where the $s-$H\"older quotient is defined as
$$
D_s u(x,y) = \frac{	u(x)-u(y)}{|x-y|^s}.
$$
Here \emph{p.v.} stands for {\em in principal value}, $s\in (0,1)$ is a fractional parameter  and $\phi=\Phi'$. This operator in naturally well-defined in the class of fractional Orlicz-Sobolev functions defined as 
$$
W^{s,\Phi}(\Omega):=\left\{ u\in L^\Phi(\Omega) \text{ such that } \rho_{s,G}(u)<\infty \right\},
$$
where the modular  $\rho_{s,\Phi}$ is  defined as
$$
\rho_{s,\Phi}(u):=
  \iint_{\R^n\times\R^n} \Phi( |D_su(x,y)|)  \,d\mu,
$$
with $d\mu(x,y):=\frac{ dx\,dy}{|x-y|^n}$. The space $W^{s,\Phi}_0(\Omega)$ is defined as the closure of the $C_c^\infty$ functions with respect to the norm $\|\cdot\|_{ W^{s,\Phi}}$. See \cite{FBS} for additional properties and its connection with the local operator $\Delta_\phi$.

In view of \cite[Proposition 2.8 and Proposition 3.8]{S}, following the arguments in the proof of Theorem \ref{main.2} with the pertinent changes, it can be obtained the following extension for the continuity of solutions with respect to the domain in the fractional case (c.f. \cite{Baroncini-Bonder-Spedaletti}).

\begin{teo}
 Assume that the design box $D\subset \R^n$ has Lipschitz boundary. Let $\Phi$ be a Young function satisfying \eqref{cond} and  \eqref{condC} such that $p^-+1>n$ and  let $f\in L^{\Phi^*}(D)$.
Assume that the sequence of domains $\{\Omega_k\}_{k\in\N}\subset D$ verifies that $\Omega_k \toh \Omega\subset D$ then $u_{\Omega_k,f}\cd u_{\Omega,f}\text{ weakly in } W^{s,\Phi}_0(D)$, where $u_k\in W^{s,\Phi}_0(\Omega_k)$ and $u \in W^{s,\Phi}_0(\Omega)$ are solution of
\begin{align}
\begin{cases}
(-\Delta_\phi)^s u_k = f &\quad \text{ in } \Omega_k\\
u_k=0 &\quad \text{ on } \partial \Omega_k
\end{cases}
\qquad \text{and} \qquad 
\begin{cases}
(-\Delta_\phi)^s u = f &\quad \text{ in } \Omega\\
u=0 &\quad \text{ on } \partial \Omega,
\end{cases}
\end{align}
respectively.
\end{teo}

\subsection{Some conjectures and open questions} \label{sec.open}

As mentioned in Remark \ref{remar}, we conjecture that the statement of Theorem \ref{main.2} can be improved up to $p^- >n$ if additionally is required that the number of connected components of each $\Omega_k^c$ is uniformly bounded. The available techniques which could lead to this result seem to be  highly nontrivial in this case, cf. \cite{Bucur-Trebeschi, HKM}.

Another interesting question we let open in this manuscript is to decide, if under the assumptions of our main results, the $\gamma-$converge of the sequence $\{\Omega_k\}_{k\in\N}$ to the limit set $\Omega$ implies that  $\lambda_\mu(\Omega_k) \to \lambda_\mu(\Omega)$ and/or $\Lambda_\mu(\Omega_k) \to \Lambda_\mu(\Omega)$, (cf. \cite[Section 2.2.3]{Henrot} and \cite[Remark 3.5.2]{HenrotPierre}) where, given a bounded and open set $A\subset \R^n$, $\lambda_\mu(A)$ denotes the \emph{Dirichlet eigenvalue of $-\Delta_\phi$ with energy} $\mu>0$, defined as the number $\lambda_\mu\in\R$ such that
$$
\int_A \frac{\phi(|\nabla u_\mu|)}{|\nabla u_\mu|}\nabla u_\mu\cdot \nabla v \,dx = \lambda_\mu \int_A \phi(|u_\mu|)v\,dx \quad \forall v\in W^{1,\Phi}_0(A),
$$
being $u_\mu \in W^{1,\Phi}_0(A)$ such that $\int_A \Phi(|u_\mu|)\,dx=\mu$ the corresponding eigenfunction; on the other hand, $\Lambda_\mu(A)$ denotes the quantity 
$$
\Lambda_\mu = \min\left\{ \frac{\int_A \Phi(|\nabla u|)\,dx}{\int_A \Phi(|u|)\,dx}\colon u\in W^{1,\Phi}_0(\Omega) \text{ and } \int_A \Phi(|u|)\,dx=\mu\right\}.
$$
This minimum is attained by the same eigenfunction $u_\mu$.
Observe that in general $\lambda_\mu$ is not variational and $\lambda_\mu$ may be different to $\Lambda_\mu$. See \cite{S}.

\section*{Acknowledgements}
All of the authors were partially supported by grants UBACyT 20020130100283BA, CONICET, PIP 11220150100032CO, PROICO 031906 (UNSL) and PROIPRO 032418 (UNSL).

Part of this manuscript was produced during a visit of the second author at UNSL, he strongly appreciates the kind hospitality of the third author.

\bibliographystyle{amsplain}
\bibliography{biblio}

\end{document}